\documentclass[a4paper]{amsart}
\usepackage{amsfonts,amssymb,amsmath,enumerate,verbatim,mathtools,tikz,bm,tikz-cd,hyperref,comment}
\usepackage{tikz-cd}
\usepackage[all,2cell]{xy}
\hypersetup{
    colorlinks=true,
    linkcolor=blue,
    filecolor=magenta,      
    urlcolor=cyan,
}

\UseAllTwocells \SilentMatrices
\newtheorem{thm}{Theorem}[section]

\newtheorem{cor}[thm]{Corollary}
\newtheorem{lem}[thm]{Lemma}
\newtheorem{exm}[thm]{Example}

\newtheorem{prop}[thm]{Proposition}
\theoremstyle{definition}
\newtheorem{defn}[thm]{Definition}
\theoremstyle{remark}

\newtheorem{rem}[thm]{\bf Remark}
\numberwithin{equation}{section}

\begin{document}
\title[Singular equivalences and quadratic monomial algebras]{Singular equivalences induced by bimodules and quadratic monomial algebras}
\author[Xiao-Wu Chen, Jian Liu, Ren Wang] {Xiao-Wu Chen, Jian Liu, Ren Wang$^*$}

\thanks{$^*$ the correspondence author}
\thanks{}
\date{\today}


\maketitle

\dedicatory{}%
\commby{}%
%\begin{center}
%\end{center}

\begin{abstract}
We investigate the problem when the tensor functor by a bimodule yields a singular equivalence. It turns out that this problem is equivalent to the one when the Hom functor given by the same bimodule induces a triangle equivalence between the homotopy categories of acyclic complexes of injective modules. We give conditions on when a bimodule appears in a pair of bimodules, that defines a singular equivalence with level. We construct an explicit bimodule, which yields a singular equivalence between a quadratic monomial algebra and its associated algebra with radical square zero.  Under certain conditions which include the Gorenstein cases, the bimodule does appear in a pair of bimodules defining  a singular equivalence with level.
\end{abstract}

\section{Introduction}
Let $k$ be a field, and  $A$ be a finite dimensional $k$-algebra. Following \cite{Buc, Or04}, the \emph{singularity category} $\mathbf{D}_{\rm sg}(A)$ of $A$ is defined to be the Verdier quotient category of the bounded derived category of $A$-modules with respect to the full subcategory of perfect complexes; see also \cite{Hap91}. The singularity category is a fundamental homological invariant for an algebra with infinite global dimension.

The \emph{dg singularity category} $\mathbf{S}_{\rm dg}(A)$ \cite{Kel18} is a canonical dg enhancement of $\mathbf{D}_{\rm sg}(A)$: it is a pretriangulated dg category whose zeroth cohomology coincides with  $\mathbf{D}_{\rm sg}(A)$.  The homotopy category $\mathbf{K}_{\rm ac}(A\mbox{-Inj})$ of acyclic complexes of injective modules is a \emph{compact-completion} of the singularity category \cite{Kra}. To be more precisely, the category $\mathbf{K}_{\rm ac}(A\mbox{-Inj})$ is compactly generated and its full subcategory of compact objects is triangle equivalent to $\mathbf{D}_{\rm sg}(A)$. As is expected, the category $\mathbf{K}_{\rm ac}(A\mbox{-Inj})$ is triangle equivalent to $\mathbf{D}(\mathbf{S}_{\rm dg}(A))$, the derived category of right dg $\mathbf{S}_{\rm dg}(A)$-modules.

Let $B$ be another finite dimensional algebra. By a \emph{singular equivalence} between $A$ and $B$, we mean a triangle equivalence between $\mathbf{D}_{\rm sg}(A)$ and $\mathbf{D}_{\rm sg}(B)$. As in \cite{ZZ, Wan15, Da}, the question when the tensor functor by an $A$-$B$-bimodule $M$ yields a singular equivalence is of interest. We assume that $M$ is projective on each side. Indeed, this situation is not so restricted, as replacing $M$ by a bounded complex will not give rise to more functors between singularity categories; see Lemma~\ref{lem:comp}. Moreover, such a tensor functor lifts automatically to a dg functor between the dg singularity categories. The Hom functor given by $M$ induces a triangle functor between the compact-completions $\mathbf{K}_{\rm ac}(A\mbox{-Inj})$ and $\mathbf{K}_{\rm ac}(B\mbox{-Inj})$.

Let us mention the useful notion of a singular equivalence with level  in \cite{Wan15}, as a singular analogue to the well-known notion of a stable equivalence of Morita type \cite{Brou}. Recently, it is proved in \cite{CLW} that Keller's conjecture for singular Hochschild cohomology is invariant under singular equivalences with levels. Therefore, we are interested in constructing singular equivalences with levels. We mention that related results in the Gorenstein cases are obtained in \cite{Da}.

 In \cite{Chen18}, the first author constructs an explicit singular equivalence between a quadratic monomial algebra $A$ and its associated algebra $B$ with radical square zero. We are motivated by the following natural question: is the singular equivalence in \cite{Chen18}  induced by some bimodule?

  We answer the above question affirmatively by constructing an explicit $A$-$B$-bimodule $M$, which induces the mentioned singular equivalence. Moreover, under certain conditions, the bimodule $M$ does appear in a pair $(M, N)$,  which defines a singular equivalence with level; see Theorem~\ref{thm:main} and Proposition~\ref{prop:main}. Combining these results with \cite{CLW}, we conclude that Keller's conjecture holds for a certain class of quadratic monomial algebras, which include the Gorenstein cases (for example, gentle algebras \cite{GR}).

The paper is structured as follows. We study general results on singularity categories in Sections~2-4, and concentrate on quadratic monomial algebras in Sections~5-8. As indicated above, the main results are Theorem~\ref{thm:main} and Proposition~\ref{prop:main}.

 In Section 2, we prove that the homotopy category of acyclic complexes  of injective modules is triangle equivalent to the derived category of right dg modules over the dg singularity category; see Theorem~\ref{thm:Kr}. In Section 3, we prove that if a bimodule induces a singular equivelence, then its Hom functor induces a triangle equivalence between the homotopy category of acyclic complexes of injective modules; see Proposition~\ref{prop:bimod}. In Section 4, we give sufficient conditions on when a bimodule appears in a pair that defines a singular equivalence with level; see Propositions~\ref{prop:sing} and \ref{prop:sing2}.

In Section~5, we recall from \cite{Chen18} a singular equivalence between a quadratic monomial algebra $A$ and its associated algebra $B$ with radical square zero. In Section~6, we construct an explicit $A$-$B$-bimodule $M$, which realizes the mentioned singular equivalence by a tensor functor;  moreover, in the Gorenstein cases, we obtain a singular equivalence with level; see Theorem~\ref{thm:main}. In Section~7, by analyzing the $A$-dual bimodule of $M$, we go beyond the Gorenstein cases in Proposition~\ref{prop:main}. In the final section, we study the $B$-dual bimodule of $M$.

We work over a fixed field $k$, that is, we require that all categories and functors are $k$-linear. By default, modules mean left modules. For any bimodule, we require that $k$ acts centrally.

\section{Singularity categories and related categories}

In this section, we recall basic facts on singularity categories and dg singularity categories. We prove that the homotopy category of acyclic complexes of injective modules is triangle equivalent to the derived category of right dg modules over the dg singularity category. This result is known to experts.  Throughout, we fix a left noetherian $k$-algebra $A$.

Denote by $A\mbox{-mod}$ the abelian category of finitely generated left $A$-modules, and by $A\mbox{-proj}$ the full subcategory formed by projective modules.  Denote by $\mathbf{D}^b(A\mbox{-mod})$ its bounded derived category. By convention, an $A$-module is viewed as a stalk complex concentrated in degree zero.

Recall that an object $X$ in $\mathbf{D}^b(A\mbox{-mod})$ is called a \emph{perfect complex}, provided that it is isomorphic to a bounded complex of finitely generated projective $A$-modules. Denote by ${\bf per}(A)\subseteq \mathbf{D}^b(A\mbox{-mod})$ the full subcategory formed by perfect complexes; it is a thick triangulated subcategory. Moreover, the quotient  functor $\mathbf{K}^b(A\mbox{-mod})\rightarrow \mathbf{D}^b(A\mbox{-mod})$ induces a triangle equivalence
$$\mathbf{K}^b(A\mbox{-proj})\stackrel{\sim}\longrightarrow {\bf per}(A).$$

Following \cite{Buc, Or04}, the \emph{singularity category} of $A$ is defined to be the following Verdier quotient
$$\mathbf{D}_{\rm sg}(A)=\mathbf{D}^b(A\mbox{-mod})/{{\bf per}(A)}.$$
The terminology is justified by the following fact: $\mathbf{D}_{\rm sg}(A)$ vanishes if and only if each finitely generated $A$-module has finite projective dimension.

Denote by $A\mbox{-\underline{mod}}$ the stable category of $A\mbox{-mod}$ modulo morphisms factoring through projective modules. There is a \emph{canonical functor}
\begin{align}\label{equ:can}
{\rm can}\colon A\mbox{-\underline{mod}}\longrightarrow \mathbf{D}_{\rm sg}(A)\end{align}
sending a module to the corresponding stalk complex concentrated in degree zero. It is well defined since projective modules are isomorphic to zero in $\mathbf{D}_{\rm sg}(A)$.

For an additive category $\mathcal{A}$, we denote by $\mathcal{A}^\natural$ its \emph{idempotent completion} \cite{BS}. The canonical embedding $\iota\colon \mathcal{A}\hookrightarrow \mathcal{A}^\natural$ is dense if and only if $\mathcal{A}$ is idempotent-complete.  Each additive functor $F\colon \mathcal{A}\rightarrow \mathcal{A}'$ induces in a straightforward manner an additive functor $F^\natural \colon \mathcal{A}^\natural\rightarrow {\mathcal{A}'}^\natural$.

If $\mathcal{A}$ is triangulated, then $\mathcal{A}^\natural$ is uniquely triangulated such that $\iota$ is a triangle functor \cite{BS}.  In general, $\mathbf{D}_{\rm sg}(A)$ is not idempotent-complete \cite{Or11}. However, if $A$ is finite dimensional over $k$, then  $\mathbf{D}_{\rm sg}(A)$ is idempotent-complete; see \cite[Corollary~2.4]{Chen11}.

Let $\mathcal{T}$ be a triangulated category with arbitrary coproducts. As usual, we denote by $\Sigma$ the translation functor of $\mathcal{T}$. An object $X$ is \emph{compact} if ${\rm Hom}_\mathcal{T}(X, -)$ commutes with arbitrary coproducts. Denote by $\mathcal{T}^c$ the full subcategory formed by compact objects; it is a thick triangulated subcategory. In particular, $\mathcal{T}^c$ is always idempotent-complete.  The triangulated category $\mathcal{T}$ is \emph{compactly generated}, provided that  there is a set $\mathcal{S}$ of compact objects such that each nonzero object $X$ satisfies ${\rm Hom}_\mathcal{T}(\Sigma^i(S), X)\neq 0$ for some $S\in \mathcal{S}$ and $i\in \mathbb{Z}$.

For a small triangulated category $\mathcal{A}$, its \emph{compact-completion} means a compactly generated triangulated category $\mathcal{T}$ with a triangle embedding $\mathcal{A}\hookrightarrow \mathcal{T}$ which induces a triangle equivalence
$$\mathcal{A}^
\natural \stackrel{\sim}\longrightarrow \mathcal{T}^c.$$ The uniqueness of  compact-completions is not known in general; compare \cite[Section~2]{Kra}.

 Let $\mathcal{C}$ be a small dg category \cite{Kel94, Dri}. Its \emph{homotopy category} $H^0(\mathcal{C})$ is defined to be a category with the same objects as $\mathcal{C}$ such that its Hom spaces are the zeroth cohomology of the corresponding Hom complexes in $\mathcal{C}$.  We denote by $\mathbf{D}(\mathcal{C})$ the derived category of \emph{right} dg $\mathcal{C}$-modules. Then we have the Yoneda embedding
$$\mathbf{Y}\colon H^0(\mathcal{C})\longrightarrow \mathbf{D}(\mathcal{C}), \quad C\mapsto \mathcal{C}(-, C).$$
Recall that $\mathbf{D}(\mathcal{C})$ is compactly generated such that the smallest thick subcategory containing in the essential image of $\mathbf{Y}$ coincides with the full subcategory $\mathbf{D}(\mathcal{C})^c$ of compact objects.

Recall that a dg category $\mathcal{C}$ is \emph{pretriangulated}, provided that the essential image of $\mathbf{Y}$ is a triangulated subcategory. In this situation, the homotopy category $H^0(\mathcal{C})$ inherits a canonical triangulated structure. Then $\mathbf{Y}$ induces a triangle equivalence
 \begin{align}\label{equ:pretri}
 H^0(\mathcal{C})^\natural \stackrel{\sim}\longrightarrow\mathbf{D}(\mathcal{C})^c.
 \end{align}
 In other words, the Yoneda embedding yields a canonical compact-completion of $H^0(\mathcal{C})$.

For a full dg subcategory $\mathcal{D}$ of $\mathcal{C}$, we denote by $\mathcal{C}/\mathcal{D}$ the corresponding dg quotient. Since we work over a field, the dg category $\mathcal{C}/\mathcal{D}$ is  simply constructed from $\mathcal{C}$ by freely adding new morphisms $\varepsilon_{D}\colon D\rightarrow D$ of degree $-1$ for each object $D$ in $\mathcal{D}$,  such that $d(\varepsilon_D)=1_D$; see \cite[Subsection~3.1]{Dri} and compare \cite[Section~4]{Kel99}. Denote by $q\colon \mathcal{C}\rightarrow \mathcal{C}/\mathcal{D}$ the quotient functor, which acts on objects by the identity.

The following results summarize basic properties of dg quotient functors.

\begin{thm}\label{thm:quo}
Keep the notation as above. Then the following statements hold.
\begin{itemize}
\item[(1)] The natural functor $\mathbf{D}(\mathcal{C}/\mathcal{D})\longrightarrow \mathbf{D}(\mathcal{C})$, sending $M$ to $Mq$, is fully faithful; moreover, a right dg $\mathcal{C}$-module $X$ lies in the essential image if and only if $X(D)$ is acyclic for each $D\in \mathcal{D}$.
    \item[(2)] Assume that both $\mathcal{C}$ and $\mathcal{D}$ are pretriangulated. Then $\mathcal{C}/\mathcal{D}$ is pretriangulated. Moreover, the quotient functor $q$ induces  a triangle equivalence
        $$H^0(\mathcal{C})/H^0(\mathcal{D})\stackrel{\sim}\longrightarrow H^0(\mathcal{C}/\mathcal{D}).$$
\end{itemize}
\end{thm}

\begin{proof}
The first result is contained in \cite[Section~4]{Kel99} and \cite[Proposition~4.6]{Dri}, and the second one is a direct consequence of \cite[Theorem~3.4]{Dri}.
\end{proof}

 For two complexes $X=(X^n, d_X^n)_{n\in \mathbb{Z}}$ and $Y=(Y^n, d_Y^n)_{n\in \mathbb{Z}}$ of $A$-modules, the Hom complex ${\rm Hom}_A(X, Y)$ is given such that
$${\rm Hom}_A(X, Y)^n=\prod_{p}{\rm Hom}_A(X^p, Y^{p+n})$$
with differential
$$d(f)=d_Y\circ f-(-1)^{|f|} f\circ d_X.$$
Here, $|f|$ denotes the degree of $f$. This defines the dg category $C_{\rm dg}^b(A\mbox{-mod})$  of bounded complexes of $A$-modules; it is pretriangulated. We observe that
$$H^0(C_{\rm dg}^b(A\mbox{-mod}))=\mathbf{K}^b(A\mbox{-mod})$$
 as triangulated categories. In other words, the dg category $C_{\rm dg}^b(A\mbox{-mod})$ is a canonical dg enhancement of the usual homotopy category $\mathbf{K}^b(A\mbox{-mod})$.

We recall from \cite[Subsection~9.8]{Kel05}  a canonical dg enhancement of  $\mathbf{D}^b(A\mbox{-mod})$.  Consider the full dg subcategory $C^{b, {\rm ac}}_{\rm dg}(A\mbox{-mod})$ of  $C_{\rm dg}^b(A\mbox{-mod})$ formed by acyclic complexes. The \emph{bounded dg derived category} of $A\mbox{-mod}$ is defined to be
the dg quotient
$$\mathbf{D}_{\rm dg}^b(A\mbox{-mod})=C^b_{\rm dg}(A\mbox{-mod})/{C^{b, {\rm ac}}_{\rm dg}(A\mbox{-mod})}.$$
In view of Theorem~\ref{thm:quo}(2), the bounded dg derived category $\mathbf{D}_{\rm dg}^b(A\mbox{-mod})$ is pretriangulated, and  there is a canonical isomorphism of  triangulated categories
$$ \mathbf{D}^b(A\mbox{-mod})\stackrel{\sim}\longrightarrow H^0(\mathbf{D}_{\rm dg}^b(A\mbox{-mod})),$$
which acts on objects by the identity. Consequently, we have a canonical compact-completion
 \begin{align*}
\mathbf{D}^b(A\mbox{-mod}) \stackrel{\sim}\longrightarrow  H^0(\mathbf{D}_{\rm dg}^b(A\mbox{-mod})) \stackrel{\mathbf{Y}}\longrightarrow \mathbf{D}(\mathbf{D}_{\rm dg}^b(A\mbox{-mod})),
\end{align*}
where $\mathbf{Y}$ denotes the Yoneda embedding. Indeed, it induces a triangle equivalence
\begin{align}\label{equ:cc-D}
\mathbf{D}^b(A\mbox{-mod}) \stackrel{\sim}\longrightarrow \mathbf{D}(\mathbf{D}_{\rm dg}^b(A\mbox{-mod}))^c.
\end{align}
Here, we use the fact that $\mathbf{D}^b(A\mbox{-mod})$ is idempotent-complete; see \cite[Corollary~2.10]{BS}.

Consider $\mathbf{per}_{\rm dg}(A)$ the full dg subcategory of  $\mathbf{D}_{\rm dg}^b(A\mbox{-mod})$ formed by perfect complexes. It is natural to define the \emph{dg singularity category} \cite{Kel18} of $A$ as the following dg quotient
$$\mathbf{S}_{\rm dg}(A)=\mathbf{D}_{\rm dg}^b(A\mbox{-mod})/{{\bf per}_{\rm dg}(A)}.$$
By the same reasoning as above, we have a canonical isomorphism of triangulated categories
$$\mathbf{D}_{\rm sg}(A)\stackrel{\sim}\longrightarrow H^0(\mathbf{S}_{\rm dg}(A)).$$
Consequently, we have a canonical compact-completion of the singularity category
$$\mathbf{D}_{\rm sg}(A)\stackrel{\sim}\longrightarrow H^0(\mathbf{S}_{\rm dg}(A)) \stackrel{\bf Y}\longrightarrow \mathbf{D}(\mathbf{S}_{\rm dg}(A)). $$
It induces a triangle equivalence
\begin{align}
\mathbf{D}_{\rm sg}(A)^\natural \stackrel{\sim}\longrightarrow \mathbf{D}(\mathbf{S}_{\rm dg}(A))^c.
\end{align}

We recall from \cite{Kra} another compact-completion of the singularity category. For this, we denote by $A\mbox{-Inj}$ the category of all injective $A$-modules. Denote by $\mathbf{K}(A\mbox{-Inj})$ the homotopy category of unbounded complexes of injective $A$-modules, and by $\mathbf{K}_{\rm ac}(A\mbox{-Inj})$ the full subcategory formed by acyclic complexes. The following triangle functor is well defined
$$\Phi\colon \mathbf{K}(A\mbox{-Inj})\longrightarrow \mathbf{D}(C^b_{\rm dg}(A\mbox{-mod})), \quad I\mapsto {\rm Hom}_A(-, I).$$

In view of \cite[Theorem~4.3]{Kel94}, the following results are expected by experts. The left equivalence is due to \cite[Proposition~A.1]{Kra} in a slightly different form, and the right one is suggested by \cite[Corollary~5.4]{Kra}.

\begin{thm}\label{thm:Kr}
The above triangle functor $\Phi$ induces triangle equivalences
$$\mathbf{K}(A\mbox{-{\rm Inj}})\stackrel{\sim}\longrightarrow  \mathbf{D}(\mathbf{D}^b_{\rm dg}(A\mbox{-{\rm mod}})) \mbox{  and   } \;  \mathbf{K}_{\rm ac}(A\mbox{-{\rm Inj}}) \stackrel{\sim}\longrightarrow  \mathbf{D}(\mathbf{S}_{\rm dg}(A)).$$
\end{thm}

\begin{proof}
Denote by $\mathcal{X}$ the full subcategory of $\mathbf{D}(C^b_{\rm dg}(A\mbox{-mod}))$ formed by those dg modules $M$ such that $M(X)$ is acyclic for any bounded acyclic  complex $X$ of $A$-modules. By Theorem~\ref{thm:quo}(1), we identify $\mathbf{D}(\mathbf{D}^b_{\rm dg}(A\mbox{-{\rm mod}}))$ with $\mathcal{X}$. Recall that ${\rm Hom}_A(X, I)$ is acyclic for any  bounded-below acyclic complex $X$ and any complex $I$ of injective $A$-modules. It follows that the functor
\begin{align}\label{equ:Kr}
\Phi\colon \mathbf{K}(A\mbox{-{\rm Inj}})\longrightarrow  \mathcal{X}=\mathbf{D}(\mathbf{D}^b_{\rm dg}(A\mbox{-{\rm mod}}))
\end{align}
is well defined.

By \cite[Proposition~2.3(1)]{Kra}, the homotopy category $\mathbf{K}(A\mbox{-{\rm Inj}})$ is compactly generated. The above functor $\Phi$ respects arbitrary coproducts; moreover, by \cite[Proposition~2.3(2)]{Kra}, it restricts to an equivalence
\begin{align*}
\mathbf{K}(A\mbox{-{\rm Inj}})^c \stackrel{\sim}\longrightarrow  \mathbf{D}(\mathbf{D}^b_{\rm dg}(A\mbox{-{\rm mod}}))^c= \mathbf{D}^b(A\mbox{-{\rm mod}})
\end{align*}
between the subcategories of compact objects. Here, the rightmost equality means the canonical equivalence (\ref{equ:cc-D}).  It follows immediately that $\Phi$ is a triangle equivalence.

Denote by $\mathcal{Y}$ the full subcategory of $\mathbf{D}(\mathbf{D}^b_{\rm dg}(A\mbox{-{\rm mod}}))$ formed by those dg modules $N$ such that $N(P)$ is acyclic for any perfect complex $P$. In view of Theorem~\ref{thm:quo}(1), we identify $\mathcal{Y}$ with $\mathbf{D}(\mathbf{S}_{\rm dg}(A))$. It is well known that a complex $I$ of injective $A$-modules is acyclic if and only if $\Phi(I)(P)={\rm Hom}_A(P, I)$ is acyclic for any perfect complex $P$; compare \cite[(2.1)]{Kra}. Then the equivalence (\ref{equ:Kr}) restricts an equivalence
\begin{align}\label{equ:Kr2}
\mathbf{K}_{\rm ac}(A\mbox{-{\rm Inj}})\stackrel{\sim}\longrightarrow  \mathcal{Y}=\mathbf{D}(\mathbf{S}_{\rm dg}(A)),
\end{align}
as required.
\end{proof}

\section{Singular equivalences induced by bimodules}

In this section, we investigate the situation where a bimodule induces a tensor functor between singularity categories and a Hom functor between the homotopy categories of complexes.

Throughout, we assume that both $A$ and $B$ are left noetherian $k$-algebras. Let $M={_AM_B}$ be an $A$-$B$-bimodule, on which $k$ acts centrally. We require further that both $_AM$ and $M_B$ are finitely generated projective.

The projectivity assumption on $M$ implies that
$$M\otimes_B-\colon \mathbf{D}^b(B\mbox{-mod})\longrightarrow \mathbf{D}^b(A\mbox{-mod})$$
is well defined, which preserves perfect complexes. It induces uniquely a triangle functor
$$M\otimes_B-\colon \mathbf{D}_{\rm sg}(B)\longrightarrow \mathbf{D}_{\rm sg}(A).$$
We are interested in when this induced functor is a singular equivalence, meaning a triangle equivalence between the singularity categories.

The above two triangle functors lift to dg functors between the corresponding dg enhancements:
$$M\otimes_B-\colon \mathbf{D}^b_{\rm dg}(B\mbox{-mod})\longrightarrow \mathbf{D}^b_{\rm dg}(A\mbox{-mod}) \mbox{ and } M\otimes_B-\colon \mathbf{S}_{\rm dg}(B)\longrightarrow \mathbf{S}_{\rm dg}(A).$$
For each injective $A$-module $E$, ${\rm Hom}_A(M, E)$ is an injective $B$-module. Then we have the following well-defined triangle functors:
$${\rm Hom}_A(M, -)\colon \mathbf{K}(A\mbox{-Inj})\longrightarrow \mathbf{K}(B\mbox{-Inj}) \mbox{ and } {\rm Hom}_A(M, -)\colon \mathbf{K}_{\rm ac}(A\mbox{-Inj})\longrightarrow \mathbf{K}_{\rm ac}(B\mbox{-Inj}).$$

We apply Theorem~\ref{thm:Kr} to $A$ and $B$, and obtain the corresponding triangle equivalences $\Phi_A$ and $\Phi_B$. We claim that the following diagram is commutative.
\[\xymatrix{
\mathbf{K}(A\mbox{-Inj})\ar[d]_-{\Phi_A} \ar[rr]^-{{\rm Hom}_A(M, -)} && \mathbf{K}(B\mbox{-Inj}) \ar[d]^-{\Phi_B}\\
\mathbf{D}(\mathbf{D}^b_{\rm dg}(A\mbox{-mod})) \ar[rr]^-{(M\otimes_B-)^*}&& \mathbf{D}(\mathbf{D}^b_{\rm dg}(B\mbox{-mod}))
}\]
Here, the bottom arrow sends a right dg $\mathbf{D}^b_{\rm dg}(A\mbox{-mod})$-module $X$ to the composition $X\circ (M\otimes_B-)$, which is a right dg $\mathbf{D}^b_{\rm dg}(B\mbox{-mod})$-module. Indeed, the commutativity follows from the following standard fact: for any complex $I$ of injective $A$-modules and a bounded complex $Y$ of $B$-modules, there is a canonical isomorphism of complexes
$${\rm Hom}_B(Y, {\rm Hom}_A(M, I))\simeq {\rm Hom}_A(M\otimes_B Y, I).$$

As the equivalence (\ref{equ:Kr2}) is restricted from the equivalence (\ref{equ:Kr}), the above commutative diagram restricts to the following commutative diagram.
\begin{align}\label{comm:1}
\xymatrix{
\mathbf{K}_{\rm ac}(A\mbox{-Inj})\ar[d]_-{\Phi_A} \ar[rr]^-{{\rm Hom}_A(M, -)} && \mathbf{K}_{\rm ac}(B\mbox{-Inj}) \ar[d]^-{\Phi_B}\\
\mathbf{D}(\mathbf{S}_{\rm dg}(A)) \ar[rr]^-{(M\otimes_B-)^*}&& \mathbf{D}(\mathbf{S}_{\rm dg}(B))
}\end{align}

Recall that a dg functor $F\colon \mathcal{C}\rightarrow \mathcal{D}$ between dg categories is \emph{quasi-fully faithful}, if for any objects $C, C'\in \mathcal{C}$, the induced cochain map
$$\mathcal{C}(C, C')\longrightarrow \mathcal{D}(F(C), F(C'))$$
is a quasi-isomorphism. It follows that $H^0(F)\colon H^0(\mathcal{C})\rightarrow H^0(\mathcal{D})$ is fully faithful. A quasi-fully faithful dg functor $F$ is said to be a \emph{quasi-equivalence}, if $H^0(F)$ is dense and thus an equivalence.

The implication ``(1) $\Rightarrow$ (3)" in the following result is implicitly contained in \cite[Theorem~6.6]{Kra}.

\begin{prop}\label{prop:bimod}
Keep the assumptions as above. Consider the following statements.
\begin{itemize}
\item[(1)] The triangle functor $M\otimes_B-\colon \mathbf{D}_{\rm sg}(B)\rightarrow \mathbf{D}_{\rm sg}(A)$ is an equivalence;
    \item[(2)] The dg functor $M\otimes_B-\colon \mathbf{S}_{\rm dg}(B)\rightarrow \mathbf{S}_{\rm dg}(A)$ is a quasi-equivalence;
    \item[(3)] The triangle functor ${\rm Hom}_A(M, -)\colon \mathbf{K}_{\rm ac}(A\mbox{-{\rm Inj}})\rightarrow \mathbf{K}_{\rm ac}(B\mbox{-{\rm Inj}})$ is an equivalence;
        \item[(4)] The triangle functor $(M\otimes_B-)^\natural\colon \mathbf{D}_{\rm sg}(B)^\natural\rightarrow \mathbf{D}_{\rm sg}(A)^\natural $ is an equivalence.
\end{itemize}
Then we have implications (1)$\Leftrightarrow$ (2) $\Rightarrow$ (3) $\Leftrightarrow$ (4).
\end{prop}

\begin{proof}
Recall the identifications $H^0(\mathbf{S}_{\rm dg}(B))=\mathbf{D}_{\rm sg}(B)$ and $H^0(\mathbf{S}_{\rm dg}(A))=\mathbf{D}_{\rm sg}(A)$. Then ``(1)$\Leftrightarrow$(2)" follows from Lemma~\ref{lem:pretri}(1) below.

By the commutative diagram (\ref{comm:1}) and the equivalences $\Phi_A$ and $\Phi_B$, we infer that (3) is equivalent to the condition that
$$(M\otimes_B-)^*\colon \mathbf{D}(\mathbf{S}_{\rm dg}(A))\longrightarrow \mathbf{D}(\mathbf{S}_{\rm dg}(B))$$
is a triangle equivalence. Then  ``(3)$\Leftrightarrow$(4)"  follows from Lemma~\ref{lem:pretri}(2). The implication ``(1)$\Rightarrow$(4)" is clear.
\end{proof}

For a dg functor $F\colon \mathcal{C}\rightarrow \mathcal{D}$, we denote by $F^*\colon \mathbf{D}(\mathcal{D})\rightarrow \mathbf{D}(\mathcal{C})$ the obvious functor sending a right dg $\mathcal{D}$-module $X$ to the right dg $\mathcal{C}$-module $X\circ F$. The following general facts are well known.

\begin{lem}\label{lem:pretri}
Let $F\colon \mathcal{C}\rightarrow \mathcal{D}$ be a dg functor between two pretriangulated dg categories. Then the following statements hold.
\begin{itemize}
\item[(1)] The functor $F$ is a quasi-equivalence if and only if $H^0(F)\colon H^0(\mathcal{C})\rightarrow H^0(\mathcal{D})$ is a triangle equivalence;
    \item[(2)] The functor $F^*\colon \mathbf{D}(\mathcal{D})\rightarrow \mathbf{D}(\mathcal{C})$ is a triangle equivalence if and only if $H^0(F)^\natural\colon  H^0(\mathcal{C})^\natural\rightarrow H^0(\mathcal{D})^\natural$ is a triangle equivalence.
\end{itemize}
\end{lem}

\begin{proof}
For (1), we refer to \cite[Lemma~3.1]{CC}. For (2), we recall that $F^*$ has a left adjoint $F_*$. To be more precisely, let $X_F$ be a dg $\mathcal{C}$-$\mathcal{D}$-bimodule given by $X_F(D, C)=\mathcal{D}(D, F(C))$. Then $F_*=-\otimes_\mathcal{C}^\mathbb{L}X_F$; see \cite[Example~6.1]{Kel94}. Since $F_*$ commutes with arbitrary coproducts and  preserves compact objects, it is an equivalence if and only if so is its restriction on compact objects.

 We observe the following commutative diagram.
\[\xymatrix{
H^0(\mathcal{C}) \ar[d]_-{\mathbf{Y}}\ar[rr]^-{H^0(F)} &&  H^0(\mathcal{D}) \ar[d]^-{\mathbf{Y}}\\
\mathbf{D}(\mathcal{C}) \ar[rr]^-{F_*} &&  \mathbf{D}(\mathcal{D})
}\]
Here, the vertical arrows are the Yoneda embeddings. In view of (\ref{equ:pretri}), we conclude that the restriction of $F_*$ on compact objects coincides with $H^0(F)^\natural$.

 Therefore, $F_*$ is an equivalence if and only if so is $H^0(F)^\natural$. Finally, we are done by the fact that $F^*$ is an equivalence if and only if so is the left adjoint $F_*$.
\end{proof}

\section{Singular equivalences with levels}
In this section, we give sufficient conditions on when a bimodule appears in a pair, that defines a singular equivalence with level \cite{Wan15}.  From now on, we will assume that both $A$ and $B$ are finite dimensional $k$-algebras.

As shown in the previous section,  an $A$-$B$-bimodule $M$, which is projective on each side, yields the tensor functor $M\otimes_B-$ between the singularity categories. The following lemma shows that, up to translation, replacing modules by complexes will not enlarge the class of functors.

Let $_AX_B$ be a bounded complex of finitely generated $A$-$B$-bimodules. We assume further that the underlying complexes $_AX$ and $X_B$ are both perfect. The derived tensor functor $X\otimes_B^\mathbb{L}-\colon \mathbf{D}^b(B\mbox{-mod})\rightarrow \mathbf{D}^b(A\mbox{-mod})$ is well defined and preserves perfect complexes. Therefore, we have an induced functor
$$X\otimes_B^\mathbb{L}-\colon \mathbf{D}_{\rm sg}(B)\longrightarrow \mathbf{D}_{\rm sg}(A).$$

\begin{lem}\label{lem:comp}
Let $_AX_B$ be as above. Then there exist $n\geq 0$ and  an $A$-$B$-bimodule $M$ satisfying that both $_AM$ and $M_B$ are finitely generated projective and that there is an isomorphism
$$X\otimes_B^\mathbb{L}-\simeq \Sigma^n \circ (M\otimes_B-)$$
of triangle functors between the singularity categories.
\end{lem}

\begin{proof}
By the perfectness assumption on $X$, we may replace $X$ by a bounded complex $Y$ of the following form
$$\cdots \rightarrow 0 \rightarrow M \rightarrow P^{-n+1} \rightarrow P ^{-n+2} \rightarrow  \cdots \rightarrow P^0\rightarrow \cdots$$
such that each $P^i$ is a projective $A$-$B$-bimodule, and that $M$ is an $A$-$B$-bimodule satisfying that both $_AM$ and $M_B$ are projective.  Then we have a canonical triangle in the homotopy category of bounded complexes of $A$-$B$-bimodules
$$\xi\colon \; \tau_{>-n}(Y)\longrightarrow Y \longrightarrow \Sigma^n(M) \longrightarrow \Sigma \tau_{>-n}(Y), $$
where $\tau_{>-n}(Y)$ denotes the brutal truncation of $Y$.

For each bounded complex $Z$ of $B$-modules, we observe that $\tau_{>-n}(Y)\otimes_B Z$ is perfect, that is, isomorphic to zero in $\mathbf{D}_{\rm sg}(A)$. Applying $-\otimes_B Z$ to $\xi$, we infer an isomorphism
$$Y\otimes_B Z\simeq \Sigma^n(M)\otimes Z=\Sigma^n(M\otimes_B Z)$$
in $\mathbf{D}_{\rm sg}(A)$. By the isomorphism $X\otimes^\mathbb{L}_B Z \simeq Y\otimes_B Z$, we are done.
\end{proof}

\begin{rem}
The above triangle functor $X\otimes_B^\mathbb{L}-\colon \mathbf{D}_{\rm sg}(B)\rightarrow \mathbf{D}_{\rm sg}(A)$ clearly lifts to a morphism  $\mathbf{S}_{\rm dg}(B)\rightarrow \mathbf{S}_{\rm dg}(A)$ in $\mathbf{Hodgcat}$, the homotopy category of small dg categories \cite{Tab}. In view of the derived Morita theory \cite{Toe, CC},  the following two questions seem to be fundamental: does  any triangle functor $\mathbf{D}_{\rm sg}(B)\rightarrow \mathbf{D}_{\rm sg}(A)$ lift to $\mathbf{Hodgcat}$? How to characterize the morphism set in $\mathbf{Hodgcat}$ between dg singularity categories?
\end{rem}

\begin{lem}\label{lem:fd}
Let $M$ be an $A$-$B$-bimodule which is finitely generated projective on each side. Then $M\otimes_B- \colon \mathbf{D}_{\rm sg}(B)\rightarrow \mathbf{D}_{\rm sg}(A)$ is a triangle equivalence if and only if so is ${\rm Hom}_A(M, -)\colon \mathbf{K}_{\rm ac}(A\mbox{-{\rm Inj}})\rightarrow \mathbf{K}_{\rm ac}(B\mbox{-{\rm Inj}})$.
\end{lem}

\begin{proof}
Recall that the singularity category of a finite dimensional algebra is always idempotent-complete; see \cite[Corollary~2.4]{Chen11}. Then in Proposition~\ref{prop:bimod} applied to this situation, the conditions (1) and (4) are equivalent. Then we are done.
\end{proof}

Let us recall a nice situation, where a pair of bimodules induces a singular equivalence. We denote by $A^e=A\otimes_k A^{\rm op}$ the enveloping algebra of $A$. We identify $A$-$A$-bimodules with  left $A^e$-modules. Denote by $\Omega_{A^e}(-)$ the syzygy functor on the stable category $A^e\mbox{-\underline{mod}}$ of $A$-$A$-bimodules. The following terminology is modified from \cite[Definition~2.1]{Wan15}.

\begin{defn}\label{defn:singequi}
Let $_AM_B$ and $_BN_A$ be an $A$-$B$-bimodule and a $B$-$A$-bimodule, respectively,  and let $n\geq 0$.  The pair $(M, N)$ is said to  define a \emph{singular equivalence with level $n$},  provided that the following conditions are satisfied:
\begin{enumerate}
\item[(1)] The four one-sided modules $_AM$, $M_B$, $_BN$ and $N_A$ are all finitely generated projective.
\item[(2)]  There are isomorphisms $M\otimes_B N \simeq \Omega^n_{A^e}(A)$ and $N\otimes_A M\simeq \Omega^n_{B^e}(B)$ in $A^e\mbox{-\underline{mod}}$ and $B^e\mbox{-\underline{mod}}$, respectively. \hfill $\square$
 \end{enumerate}
 \end{defn}

Let us make simple observations. We denote by $\Omega_{A\mbox{-}B}(-)$ the syzygy functor on the stable category of $A$-$B$-bimodules.

\begin{rem}\label{rem:syzygy}
Let $(M, N)$ define a singular equivalence with level $n$. Then the following statements hold.
\begin{itemize}
\item[(1)] Both $(\Omega_{A\mbox{-}B}(M), N)$ and $(M, \Omega_{B\mbox{-}A}(N))$ define singular equivalences with level $n+1$. Here, we use isomorphisms $\Omega_{A^e}(M\otimes_B N)\simeq \Omega_{A\mbox{-}B}(M)\otimes_B N\simeq M\otimes_B\Omega_{B\mbox{-}A}(N)$ of $A$-$A$-bimodules.
\item[(2)] Assume that $(M', N)$ defines a singular equivalence with level $n$. Then $M$ and $M'$ are related such that $\Omega_{A\mbox{-}B}^n(M)$ and $\Omega_{A\mbox{-}B}^n(M')$ are isomorphic in the stable category of bimodules. Here, we just compute $M\otimes_B N\otimes_A M'$ in two different ways.
    \end{itemize}
\end{rem}

We summarize related concepts in the following remark. We denote by ${\rm rad}(A)$ the Jacobson radical of $A$.

\begin{rem}\label{rem:singequi}
\begin{enumerate}
\item[(1)] A  stable equivalence of Morita type in the sense of \cite[Definition~5.A]{Brou} becomes naturally a singular equivalence with level zero. By \cite[Theorem~2.3]{Wan15}, a derived equivalence  induces a singular equivalence with a certain level. Moreover, by \cite[Proposition~2.6]{Sk}, a singular equivalence of Morita type \cite{ZZ} induces a singular equivalence with a certain level.
        \item[(2)] For an algebra $A$ with $A/{{\rm rad}(A)}$ separable over $k$, Keller's conjecture \cite{Kel18} states that the singular Hochschild cochain complex of $A$ is isomorphic to the Hochschild cochain complex of $\mathbf{S}_{\rm dg}(A)$ on the $B_\infty$-level. By \cite[Theorem~9.4(3)]{CLW}, Keller's conjecture is invariant under singular equivalences with levels.
\end{enumerate}

\end{rem}

The following observations justify the terminology. We mention that the first half is due to \cite[Remark~2.2]{Wan15} and \cite[Proposition~4.2]{Da}.

\begin{lem}\label{lem:Wang}
Assume that $(M, N)$ defines a singular equivalence with level $n$. Then the following statements hold.
\begin{enumerate}
\item The triangle functor
$$M\otimes_B- \colon \mathbf{D}_{\rm sg}(B)\longrightarrow \mathbf{D}_{\rm sg}(A)$$
is an equivalence, whose quasi-inverse is given by $\Sigma^n\circ (N\otimes_A-)$.
\item The triangle functor
$${\rm Hom}_A(M, -)\colon \mathbf{K}_{\rm ac}(A\mbox{-{\rm Inj}})\longrightarrow \mathbf{K}_{\rm ac}(B\mbox{-{\rm Inj}})$$
is an equivalence, whose quasi-inverse is given by $\Sigma^{-n}\circ {\rm Hom}_B(N, -)$.
\end{enumerate}
\end{lem}

\begin{proof}
(1) By the same argument in the proof of Lemma~\ref{lem:comp}, we infer that there is an isomorphism in $\mathbf{D}_{\rm sg}(A)$
$$\Sigma^n \Omega_{A^e}^n(A)\otimes_A Z\simeq A\otimes_AZ=Z$$
for any bounded complex $Z$ of $A$-modules. Then we have isomorphisms
$$M\otimes_B(N\otimes_A Z)\simeq \Omega_{A^e}(A)\otimes_A Z \simeq \Sigma^{-n}(Z).$$
For the same reason, we have isomorphisms in $\mathbf{D}_{\rm sg}(B)$
 $$N\otimes_A(M\otimes_B Y)\simeq  \Sigma^{-n}(Y)$$
 for any bounded complex $Y$ of $B$-modules. Then the required result follows immediately.

 (2) Let $I$ be an arbitrary acyclic complex of injective $A$-modules. We claim that for any projective $A$-$B$-bimodule $P$, the Hom complex ${\rm Hom}_A(P, I)$ of $B$-modules is contractible. Indeed, it suffices to prove the claim for $P=A\otimes_k B$. The following isomorphism of complexes
 $${\rm Hom}_A(A\otimes_k B, I)\simeq {\rm Hom}_k(B, I)$$
 implies the required contractibility, since $I$ is contractible as a complex of $k$-modules.

Let $L$ be an $A$-$A$-bimodule which is projective on each side. Take a short exact sequence of $A$-$A$-modules
$$0\longrightarrow \Omega_{A^e}(L) \longrightarrow P \longrightarrow L\longrightarrow 0$$
with $P$ projective. We have an induced short exact sequence of complexes of injective $A$-modules.
$$0\longrightarrow {\rm Hom}_A(L, I)\longrightarrow {\rm Hom}_A(P, I)\longrightarrow {\rm Hom}_A(\Omega_{A^e}(L), I)\longrightarrow 0$$
This induced sequence corresponds to an exact triangle in $\mathbf{K}_{\rm ac}(A\mbox{-Inj})$. The above claim implies that ${\rm Hom}_A(\Omega_{A^e}(L), I)\simeq \Sigma {\rm Hom}_A(L, I)$. Inductively, we infer a natural isomorphism
 \begin{align}
 {\rm Hom}_A(\Omega^n_{A^e}(L), I)\simeq \Sigma^n {\rm Hom}_A(L, I)
 \end{align}
for each $n\geq 0$.

By Lemma~\ref{lem:fd}, both  ${\rm Hom}_A(M, -)$ and ${\rm Hom}_B(N, -)$ are equivalences. We now have natural isomorphisms of complexes
\begin{align*}
&{\rm Hom}_B(N, {\rm Hom}_A(M, I)) \simeq {\rm Hom}_A(M\otimes_B N, I)\\
&\simeq {\rm Hom}_A(\Omega_{A^e}^n(A), I)\simeq \Sigma^n {\rm Hom}_A(A, I)= \Sigma^n(I).
\end{align*}
This clearly implies the required statement on the quasi-inverse. We mention that one might give an alternative proof using (\ref{comm:1}), applied both to $M$ and $N$.
\end{proof}

In what follows, we fix an $A$-$B$-bimodule $M$, which is finitely generated projective on each side.

We will consider its $A$-dual ${\rm Hom}_A(M, A)$ and $B$-dual ${\rm Hom}_{B^{\rm op}}(M, B)$, both of which are $B$-$A$-bimodules. Here, ${\rm Hom}_{B^{\rm op}}(-, -)$ means the Hom bifunctor between right $B$-modules.

In view of Lemma~\ref{lem:comp}, the following result is  a variant of \cite[Theorem~3.6]{Da} in a slightly different setting. It might be viewed as a partial converse of Lemma~\ref{lem:Wang}.

\begin{prop}\label{prop:sing}
Suppose that both $A/{\rm rad}(A)$ and $B/{\rm rad}(B)$ are separable over $k$. Assume that ${\rm Hom}_A(M, A)$ has finite projective dimension as a left $B$-module, and that $M\otimes_B- \colon \mathbf{D}_{\rm sg}(B)\rightarrow \mathbf{D}_{\rm sg}(A)$ is an equivalence. Then there is an $B$-$A$-bimodule $N$ such that $(M, N)$ defines a singular equivalence with level.
\end{prop}

\begin{proof}
Set $N'={\rm Hom}_A(M, A)$. Assume that the projective dimension of $N'$  as a left $B$-module is $c$. Since $_AM$ is finitely generated projective, there is a canonical isomorphism of $B$-$B$-bimodules
 $${\rm can}\colon N'\otimes_A M\stackrel{\sim}\longrightarrow {\rm Hom}_A(M, M), \quad f\otimes m\mapsto (x\mapsto f(x).m).$$
The following map is a morphism of $A$-$A$-bimodules
$$\varepsilon\colon M\otimes_B N'\longrightarrow A, \quad m\otimes f\mapsto f(m).$$
Dually, there is a morphism of $B$-$B$-bimodules
$$\eta\colon B\longrightarrow N'\otimes_A M$$
such that ${\rm can}\circ \eta$ sends $b\in B$ to the right action of $b$ on $M$.

 We observe that there is an isomorphism of functors from $A\mbox{-mod}$ to $B\mbox{-mod}$:
$${\rm Hom}_A(M, -)\simeq N'\otimes_A-.$$
Consequently, we have an adjoint pair $(M\otimes_B-, N'\otimes_A-)$ between $A\mbox{-mod}$ and $B\mbox{-mod}$. Moreover, its unit is given by $\eta\otimes_B-$, and its counit is given by $\varepsilon\otimes_A-$.

The above adjoint pair induces,  in a straightforward manner,  an adjoint pair between $\mathbf{D}_{\rm sg}(B)$ and $\mathbf{D}_{\rm sg}(A)$ with the induced unit and counit; compare \cite[Lemma~1.2]{Or04}. By assumption,  $M\otimes_B- \colon \mathbf{D}_{\rm sg}(B)\rightarrow \mathbf{D}_{\rm sg}(A)$ is an equivalence. Then the induced adjoint pair gives rise to mutually inverse equivalences. In particular, both the unit and counit are isomorphisms. In other words, for any bounded complex $Y$ of $B$-modules and any bounded complex $Z$ of $A$-modules, $\eta\otimes_B Y$ and $\varepsilon \otimes_A Z$ are isomorphisms in $\mathbf{D}_{\rm sg}(B)$ and $\mathbf{D}_{\rm sg}(A)$, respectively.

Applying Lemma~\ref{lem:sepa} below to $\eta$, there is a sufficiently large $a$ such that $\eta$ induces an isomorphism
$$\Omega_{B^e}^a(B)\simeq \Omega_{B^e}^a(N'\otimes_A M)= \Omega_{B\mbox{-}A}^a(N')\otimes_A M$$
in $B^e\mbox{-\underline{mod}}$.  Similarly, there is a sufficiently large $b$ such that $\varepsilon$  induces an isomorphism
$$M\otimes_B\Omega_{B\mbox{-}A}^b(N')=\Omega_{A^e}^b(M\otimes_B N')\simeq \Omega_{A^e}^b(A)$$
in $A^e\mbox{-\underline{mod}}$. Now, we take $n={\rm max}\{a, b, c\}$. Then we conclude that $(M, \Omega_{B\mbox{-}A}^n(N'))$ defines a singular equivalence with level $n$.
 \end{proof}

The following lemma is standard.

\begin{lem}\label{lem:sepa}
Let $A$ be a finite dimensional algebra such that $A/{\rm rad}(A)$ is separable over $k$. Let $U, V$ be two $A$-$A$-bimodules such that the underlying one-sided modules all have finite projective dimension. Let $f\colon U\rightarrow V$ be a morphism of $A$-$A$-bimodules such that $f\otimes_A^\mathbb{L} Z$ is an isomorphism in $\mathbf{D}_{\rm sg}(A)$ for  any bounded complex $Z$ of $A$-modules. Then there is a sufficiently large $b\geq 0$ such that $f$ induces an isomorphism $\Omega_{A^e}^b(U)\simeq \Omega_{A^e}^b(V)$ in $A^e\mbox{-\underline{\rm mod}}$.
\end{lem}

\begin{proof}
We view $f$ as a cochain map between stalk complexes of bimodules. By \cite[Lemma~3.5]{Da}, the mapping cone of $f$ is a perfect complex of $A$-$A$-bimodules. Then it follows immediately that $f$ induces an isomorphism between the higher syzygies in the stable category of $A$-$A$-bimodules.
\end{proof}

The following result is a variant of Proposition~\ref{prop:sing}.

\begin{prop}\label{prop:sing2}
Suppose that both $A/{\rm rad}(A)$ and $B/{\rm rad}(B)$ are separable over $k$. Assume that ${\rm Hom}_{B^{\rm op}}(M, B)$ has finite projective dimension as a right $A$-module, and that $M\otimes_B- \colon \mathbf{D}_{\rm sg}(B)\rightarrow \mathbf{D}_{\rm sg}(A)$ is an equivalence. Then there is an $B$-$A$-bimodule $N$ such that $(M, N)$ defines a singular equivalence with level.
\end{prop}

\begin{proof}
Set $N''={\rm Hom}_{\rm B^{\rm op}}(M, B)$. We observe that $M\otimes_B-\simeq {\rm Hom}_B(N'', -)$. Consequently, we have an adjoint pair $(N''\otimes_A-, M\otimes_B-)$ between $A\mbox{-mod}$ and $B\mbox{-mod}$. We omit  the remaining proof, as it is almost the same as the one in Proposition~\ref{prop:sing}. We just mention that $N$ is chosen to $\Omega^{m}_{B\mbox{-}A}(N'')$ for sufficiently large $m$.
\end{proof}

\begin{rem}\label{rem:sing}
Let us assume that both the assumptions in Propositions~\ref{prop:sing} and \ref{prop:sing2} hold. By the proof of the two propositions and Remark~\ref{rem:syzygy}(2),  the two dual $B$-$A$-bimodules ${\rm Hom}_{A}(M, A)$ and ${\rm Hom}_{B^{\rm op}}(M, B)$ are related: there is a sufficiently large $n$ such that
$$\Omega_{B\mbox{-}A}^n({\rm Hom}_{A}(M, A)) \simeq \Omega_{B\mbox{-}A}^n({\rm Hom}_{B^{\rm op}}(M, B)) $$
in the stable category of $B$-$A$-bimodules.
\end{rem}

Recall that an algebra $A$ is \emph{Gorenstein} \cite{Buc, Hap91} provided that both $_AA$ and $A_A$ have finite injective dimension, or equivalently, any $A$-module has finite injective dimension if and only if it has finite projective dimension.

\begin{lem}\label{lem:Gor}
Let $A$ and $B$ be two Gorenstein algebras.  Then ${\rm Hom}_{A}(M, A)$ has finite projective dimension as a left $B$-module, and ${\rm Hom}_{B^{\rm op}}(M, B)$ has finite projective dimension as a right $A$-module.
\end{lem}

\begin{proof}
The functor  ${\rm Hom}_A(M, -)\colon A\mbox{-mod}\rightarrow B\mbox{-mod}$ is exact and preserves injective modules. It follows that it preserves modules of finite injective dimension. Since $_AA$ has finite injective dimension, so does the left $B$-module ${\rm Hom}_A(M, A)$. As $B$ is Gorenstein, then ${\rm Hom}_A(M, A)$ has finite projective dimension. This proves the first half, and  the second half is similar.
\end{proof}

In view of Lemma~\ref{lem:Gor} and Remark~\ref{rem:sing}, we have the following immediate consequence; compare \cite[Theorem]{Da}.

\begin{cor}\label{cor:Gor}
Let $A$ and $B$ be two Gorenstein algebras such that both $A/{{\rm rad}(A)}$ and $B/{{\rm rad}(B)}$ are separable over $k$. Assume that $M\otimes_B- \colon \mathbf{D}_{\rm sg}(B)\rightarrow \mathbf{D}_{\rm sg}(A)$ is an equivalence. Then there is a sufficiently large $n$ such that there is an isomorphism
$$N:=\Omega_{B\mbox{-}A}^n({\rm Hom}_{A}(M, A)) \simeq \Omega_{B\mbox{-}A}^n({\rm Hom}_{B^{\rm op}}(M, B)) $$
in the stable category of $B$-$A$-bimodules and that $(M, N)$ defines a singular equivalence with level $n$. \hfill $\square$
\end{cor}

\section{Quadratic monomial algebras and relation quivers}
In this section, we recall from \cite{Chen18} a singular equivalence between a quadratic monomial algebra and its associated algebra with radical square zero.

We fix a finite quiver $Q=(Q_0, Q_1; s,t)$. Here, $Q_0$ denotes the finite set of vertices, $Q_1$ denotes the finite set of arrows, and  $s,t\colon Q_1\rightarrow Q_0$ are maps which assign to each arrow $\alpha$ its starting vertex $s(\alpha)$ and its terminating vertex $t(\alpha)$.

 A path $p$ of length $n$ in $Q$ is a sequence $p=\alpha_n\cdots \alpha_2\alpha_1$ of arrows such that $s(\alpha_i)=t(\alpha_{i-1})$ for $2\leq i\leq n$; moreover, we define its starting vertex $s(p)=s(\alpha_1)$ and its terminating vertex $t(p)=t(\alpha_n)$.  We identify a path of length one with  an arrow. To each vertex $i$, we associate a trivial path $e_i$ of length zero, and set $s(e_i)=i=t(e_i)$.

For two paths $p$ and $q$ with $s(p)=t(q)$, we write $pq$ for their concatenation. As convention, we have $p=pe_{s(p)}=e_{t(p)}p$. For two paths $p$ and $q$ in $Q$, we say that $q$ is a \emph{sub-path} of $p$ provided that $p=p''qp'$ for some paths $p''$ and $p'$.

The path algebra $kQ$ is defined as follows. As a $k$-vector space, it has a basis consisting of all the paths in $Q$. For two paths $p$ and $q$, their multiplication is given by the concatenation $pq$ if $s(p)=t(q)$; it is zero, otherwise.  The unit of $kQ$ equals  $\sum_{i\in Q_0}e_i$.

Denote by $J$ the two-sided ideal of $kQ$ generated by arrows. Then $J^d$ is spanned by all the paths of length at least $d$ for each $d\geq 2$. A two-sided ideal $I$ of $kQ$ is \emph{admissible},  provided that $J^d\subseteq I\subseteq J^2$ for some $d\geq 2$. In this case, the quotient algebra $A=kQ/I$ is finite-dimensional.

We recall that an admissible ideal $I$ of $kQ$ is quadratic monomial provided that it is generated by some paths of length  two. In this case, the quotient algebra $A=kQ/I$ is called a \emph{quadratic monomial algebra}.

In what follows, $A=kQ/I$ is a fixed quadratic monomial algebra. We denote by $\mathbf{F}$ the set of paths of length two contained in $I$. Here, the letter ``F" stands for forbidden paths.

As usual, a path $p$ in $Q$ is \emph{nonzero} in $A$,  provided that it does not belong to $I$, or equivalently, $p$ does not contain a sub-path in $\mathbf{F}$. In this case, we will abuse the image $p+I$ in $A=kQ/I$ with $p$. Therefore, the set of nonzero paths forms a $k$-basis for $A$.

For each nonzero path $p$, we consider the left ideal $Ap$ generated by $p$, which has a $k$-basis given by the nonzero paths $q$ such that $q=q'p$ for some path $q'$. We observe that for a vertex $i$, $Ae_i$ is an indecomposable projective $A$-module.   Then we have a projective cover $\pi_p\colon Ae_{t(p)}\rightarrow Ap$ sending $e_{t(p)}$ to $p$.

The following fact is contained in \cite[Lemma~4.1(2)]{Chen18}: for an arrow $\alpha$, we have  an exact sequence of $A$-modules
\begin{align}\label{equ:A}
 \bigoplus_{\{\beta\in Q_1\; |\; \beta\alpha \in \mathbf{F}\}} Ae_{t(\beta)} \stackrel{\beta}\longrightarrow Ae_{t(\alpha)}\stackrel{\pi_\alpha}\longrightarrow A\alpha \longrightarrow 0,
\end{align}
where for each $\beta$ in the index set, $\beta\colon A e_{t(\beta)}\rightarrow A e_{t(\alpha)}$ means the $A$-module morphism sending  $q$ to $q\beta$, that is, the multiplication by $\beta$ from the right.

 The \emph{relation quiver} $\mathcal{R}$ of $A=kQ/I$ is defined as follows:  its vertices are given by arrows in $Q$, and there is an arrow $[\beta \alpha]$ from the vertex $\alpha$ to the vertex $\beta$ for each element $\beta\alpha$ in $\mathbf{F}$; see \cite[Definition~5.2]{CSZ}.

Consider the corresponding algebra $B=k\mathcal{R}/J^2$ with radical square zero. Then $B$ has a $k$-basis given by
$$\{e_\alpha\; |\; \alpha\in Q_1\} \cup \{ [\beta\alpha]\;|\; \beta\alpha\in \mathbf{F}\}.$$
Its multiplication is completely determined by the following identities:
$$e_\alpha\cdot  e_\beta=\delta_{\alpha, \beta}\; e_{\alpha}, \; [\beta\alpha]=e_{\beta} \cdot [\beta\alpha]\cdot e_\alpha, \mbox{ and } [\beta\alpha] \cdot [\beta'\alpha']=0.$$
Here, $\delta_{\alpha, \beta}$ is the Kronecker symbol. The algebra $B$ is said to \emph{associated} to $A$.

For each vertex $\alpha$ in $\mathcal{R}$, we denote by $S_\alpha$ the corresponding simple $B$-module. The projective cover $Be_\alpha\rightarrow S_\alpha$ annihilates the radical of $Be_\alpha$, namely the $B$-submodule spanned by $\{[\beta\alpha]\; |\; \beta\in Q_1, \beta\alpha\in \mathbf{F}\}$. Therefore, we have the following projective presentation
\begin{align}\label{equ:B}
 \bigoplus_{\{\beta\in Q_1\; |\; \beta\alpha \in \mathbf{F}\}} Be_\beta \stackrel{[\beta\alpha]}\longrightarrow Be_\alpha\longrightarrow S_\alpha \longrightarrow 0,
\end{align}
where for each $\beta$ in the index set,  $[\beta\alpha]\colon Be_\beta\rightarrow Be_\alpha$ denotes the unique $B$-module morphism sending $e_\beta$ to $[\beta\alpha]$; compare \cite[(4.3)]{Chen18}.

The following result is implicitly contained in \cite{Chen18}.

\begin{prop}\label{prop:Chen18}
Keep the notation as above. Then there is a unique triangle functor $F\colon \mathbf{D}_{\rm sg}(B)\rightarrow \mathbf{D}_{\rm sg}(A)$ satisfying $F(S_\alpha)\simeq A\alpha$ for each $\alpha\in Q_1$; moreover, such a functor is necessarily an equivalence.
\end{prop}

\begin{proof}
Since such a triangle equivalence $F$ is constructed in \cite[Theorem~4.5]{Chen18}, it suffices to prove the uniqueness. Denote by $B\mbox{-\underline{ssmod}}$ the full subcategory of $B\mbox{-\underline{mod}}$ consisting of semisimple modules. There is a unique $k$-linear functor
$$H\colon B\mbox{-\underline{ssmod}}\longrightarrow A\mbox{-\underline{mod}}$$
 sending each $S_\alpha$ to $A\alpha$; compare \cite[Lemma~4.6]{Chen18}. Here, we implicitly use the following fact: if a simple $B$-module $S_\alpha$ is projective, then the $A$-module $A\alpha$ is also projective.

 The assumption on $F$ implies that the following diagram is commutative.
 \[\xymatrix{
 B\mbox{-\underline{ssmod}}\ar[d]_-{{\rm can}_B} \ar[rr]^{H} && A\mbox{-\underline{mod}}\ar[d]^-{{\rm can}_A}\\
 \mathbf{D}_{\rm sg}(B) \ar[rr]^-{F} && \mathbf{D}_{\rm sg}(A)
 }\]
 Here, the vertical arrows are the canonical functors in (\ref{equ:can}). It is well known that ${\rm can}_A$ identifies $\mathbf{D}_{\rm sg}(A)$ with the stabilization $\mathcal{S}(A\mbox{-\underline{mod}})$ of $A\mbox{-\underline{mod}}$; see \cite[Lemma~3.1]{Chen18}. Since syzygies of any $B$-module are semisimple, ${\rm can}_B$ identifies $\mathbf{D}_{\rm sg}(B)$ with the stabilization $\mathcal{S}(B\mbox{-\underline{ssmod}})$ of $B\mbox{-\underline{ssmod}}$; compare \cite[Corollary~2.3]{Chen18}. Therefore, applying  the universal property of stabilization,  the above commutative diagram implies that $F$ is identified with $$\mathcal{S}(H)\colon \mathcal{S}(B\mbox{-\underline{ssmod}})\longrightarrow \mathcal{S}(A\mbox{-\underline{mod}}), $$
  known as the stabilization of $H$; see \cite[Section~2]{Chen18}. This implies that $F$ is unique.
\end{proof}

\section{An explicit bimodule}\label{sec:M}

Let $A=kQ/I$ be a quadratic monomial algebra with  $\mathcal{R}$ its relation quiver. As in the previous section,  $B=k\mathcal{R}/J^2$ denotes the associated algebra with radical square zero. In this section,  we will construct an explicit $A$-$B$-bimodule $M$, which realizes the singular equivalence in Proposition~\ref{prop:Chen18} by  a tensor functor.

Consider the following set
$$X=\{(p, \alpha)\; |\; \alpha\in Q_1, p \mbox{ a nonzero path in } A \mbox{ satisfying } s(p)=t(\alpha)\}.$$
Here, the nonzero path $p$ is allowed to be trivial, that is, $(e_{t(\alpha)},\alpha)$ lies in $X$. Set $M=kX$ to be the $k$-vector space with $X$ its basis.

The left $A$-action on $M$ is naturally given by the concatenation of paths on the left component $p$. More precisely, for any nonzero path $q$ in $A$, we have
 \begin{equation*}
 q.(p, \alpha)=
 \begin{cases}
 (qp, \alpha) & \mbox{ if } s(q)=t(p) \mbox{ and } qp \mbox{ is a nonzero path in } A;\\
 0& \mbox{ otherwise.}
 \end{cases}
 \end{equation*}

The right $B$-action on $M$ is given such that $(p, \alpha).e_{\beta}=\delta_{\alpha, \beta}\; (p, \alpha)$, where $\delta_{\alpha, \beta}$ is the Kronecker symbol. Moreover, for $(p, \alpha)\in X$ and $\alpha\beta\in \mathbf{F}$, we have
\begin{equation*}
(p, \alpha).[\alpha\beta]=\begin{cases} (p\alpha, \beta) & \mbox{ if } p\alpha \mbox{ is a nonzero path in } A;\\
0 & \mbox{ otherwise.} \end{cases}
\end{equation*}
This action might be visualized as follows:
\[\xymatrix@C=5pt{
&&&&\ar@{~>}[llll]_p  & && \ar[lll]_{\alpha} &  \ar@/_/@{.}[ll] &&\ar[lll]_{\beta}
}\]
Here, we use the dotted curve to indicate that $\alpha\beta$ lies in $\mathbf{F}$, and the wavy arrow to indicate that  $p$ is a path, not necessarily an arrow. We observe that if $\beta\gamma\in \mathbf{F}$, we necessarily have
$$((p, \alpha).[\alpha\beta]).[\beta\gamma]=0.$$
The above action indeed defines a right $B$-module structure on $M$. Then we obtain the required $A$-$B$-bimodule $M$.

\begin{lem}\label{lem:M}
The above $A$-$B$-bimodule $M$ is projective on each side. Moreover, for each $\alpha\in Q_1$, we have an isomorphism $M\otimes_B S_\alpha\simeq A\alpha$ of $A$-modules.
\end{lem}

\begin{proof}
For each $\alpha\in Q_1$, we denote by $X_\alpha$ the subset of $X$ formed by elements of the form $(p, \alpha)$. We observe that each $kX_\alpha$ is an $A$-submodule of $M$ and that there is an isomorphism of $A$-modules
\begin{align}\label{iso:X}
kX_\alpha\stackrel{\sim}\longrightarrow Ae_{t(\alpha)},\quad (p, \alpha)\mapsto p.
\end{align}
 The disjoint union $X=\bigcup_{\alpha\in Q_1} X_\alpha$ yields a decomposition
\begin{align}\label{decom:M}
M=\bigoplus_{\alpha\in Q_1} kX_\alpha
\end{align}
 of $A$-modules. This proves that  $M$ is a  projective left $A$-module.

Set $D=\{(p, \alpha)\in X\; |\; p\alpha \mbox{ is a nonzero path in } A\}$. For each element $(p, \alpha)\in D$, we set
$$_{(p, \alpha)}X=\{(p, \alpha), (p\alpha, \beta)\; |\; \alpha\beta\in \mathbf{F}\}\subseteq X.$$
Similarly, each $k{_{(p, \alpha)}X}$ is a $B$-submodule of $M$ and there is an isomorphism of right $B$-modules
$$k(_{(p, \alpha)}X)\stackrel{\sim}\longrightarrow e_\alpha B, \quad (p, \alpha)\mapsto e_\alpha, (p\alpha, \beta)\mapsto [\alpha\beta].$$
The disjoint union $X=\bigcup_{(p, \alpha)\in D}\;  {_{(p, \alpha)}X}$ yields a decomposition
\begin{align}\label{decom:M2}
M=\bigoplus_{(p, \alpha)\in D} k({_{(p, \alpha)}X})
\end{align}
of right $B$-modules. It follows that $M$ is a projective   right $B$-module.

We observe that $M.e_\alpha=kX_\alpha$. Therefore, we have the following composition of isomorphisms
$$\xi_\alpha\colon M\otimes_B Be_\alpha \stackrel{\sim}\longrightarrow  M.e_\alpha=kX_\alpha \stackrel{\sim}\longrightarrow  Ae_{t(\alpha)},$$
which sends $(p, \alpha)\otimes e_\alpha$ to $p$. Moreover, for $\beta\alpha\in \mathbf{F}$, we have a commutative diagram
\[\xymatrix{
M\otimes_B Be_\beta \ar[d]_-{M\otimes_B{[\beta\alpha]}}  \ar[rr]^-{\xi_\beta} && Ae_{t(\beta)} \ar[d]^-{\beta}\\
M\otimes_B Be_\alpha  \ar[rr]^-{\xi_\alpha} && Ae_{t(\alpha)}
}\]

Applying $M\otimes_B-$ to (\ref{equ:B}) and using the above commutative diagram, we observe that the resulted exact sequence is identified with (\ref{equ:A}). Then we infer the required isomorphism $M\otimes_B S_\alpha\simeq A\alpha$.
\end{proof}

Since the $A$-$B$-bimodule $M$ is projective on each side, the following triangle functors
$$M\otimes_B-\colon \mathbf{D}_{\rm sg}(B)\longrightarrow \mathbf{D}_{\rm sg}(A) \mbox{ and } {\rm Hom}_A(M, -)\colon \mathbf{K}_{\rm ac}(A\mbox{-{\rm Inj}})\longrightarrow \mathbf{K}_{\rm ac}(B\mbox{-{\rm Inj}}) $$
are well defined.

\begin{thm}\label{thm:main}
Let $A=kQ/I$ be a quadratic monomial algebra  and $B=k\mathcal{R}/J^2$ be the associated algebra with radical square zero. Consider the above $A$-$B$-bimodule $M$. Then the following statements hold.
\begin{enumerate}
\item The above functors $M\otimes_B-$ and ${\rm Hom}_A(M, -)$ are triangle equivalences.
\item Assume that $A$ is Gorenstein. Then there is a $B$-$A$-bimodule $N$ such that $(M, N)$ defines a singular equivalence with  level.
\end{enumerate}
\end{thm}

\begin{proof}
By Lemma~\ref{lem:M}, the functor $M\otimes_B-\colon \mathbf{D}_{\rm sg}(B)\rightarrow \mathbf{D}_{\rm sg}(A)$ sends $S_\alpha$ to $A\alpha$. By the uniqueness statement of Proposition~\ref{prop:Chen18}, we infer that $M\otimes_B-$ is a triangle equivalence. In view of Lemma~\ref{lem:fd}, we are done with (1).

For (2), we observe that $B$ is also Gorenstein; see \cite[Remark~4.7]{Chen18}. Indeed, by \cite[Propostion~5.5(1)]{CSZ}, the relation quiver $\mathcal{R}$ consists of basic cycles and acyclic components. Then the algebra $B$ is a direct product of selfinjective algebras and algebras with finite global dimension, and thus is Gorenstein. Clearly, both $A/{{\rm rad}(A)}$ and $B/{{\rm rad}(B)}$ are separable over $k$. Then (2) follows immediately from Corollary~\ref{cor:Gor}.
\end{proof}

\begin{rem}
Since a gentle algebra is  Gorenstein quadratic monomial \cite{GR}, we infer from (2) that there is a singular equivalence with level between a gentle algebra and its associated algebra of radical square zero. We will go beyond the Gorenstein cases in Proposition~\ref{prop:main} below.
\end{rem}

\section{The $A$-dual bimodule}

Let $M$ be the $A$-$B$-bimodule defined in the previous section. We will study the $A$-dual bimodule ${\rm Hom}_A(M, A)$. In Proposition~\ref{prop:main}, a combinatorial condition is given on when  ${\rm Hom}_A(M, A)$ has finite projective dimension as a left $B$-module.

Let us first define an explicit $B$-$A$-bimodule. Consider the following set
$$Y=\{(\alpha|q)\; |\; \alpha\in Q_1, q \mbox{ is a nonzero path in } A \mbox{ satisfying } t(q)=t(\alpha)\}.$$
Here, the path $q$ is allowed to be trivial, that is, $(\alpha|e_{t(\alpha)})$ lies in $Y$. Set $kY$ to be the  $k$-vector space with a basis $Y$.

The right $A$-action on $kY$ is naturally given by the concatenation of paths on $q$. More precisely, for a nonzero path $p$ in $A$, we have
\begin{equation*}
(\alpha|q).p=\begin{cases} (\alpha|qp) & \mbox{ if } t(p)=s(q) \mbox{ and } qp \mbox{ is a  nonzero path in } A;\\
0 & \mbox{ otherwise.}\end{cases}
\end{equation*}
The left $B$-action on $kY$ is given such that $e_\beta.(\alpha|q)=\delta_{\beta, \alpha} (\alpha|q)$. For each $\beta\alpha\in \mathbf{F}$, we define
\begin{equation*}
[\beta\alpha].(\alpha|q)=\begin{cases} (\beta|\beta q) & \mbox{ if } \beta q \mbox{ is a nonzero path in } A;\\
0 & \mbox{ otherwise.}\end{cases}
\end{equation*}
This action might be visualized as follows.
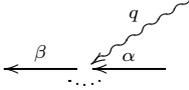
\[\xymatrix@C=5pt{
&&&&& &&\ar@{~>}[dllll]_-q \\
& && \ar[lll]_{\beta} &  \ar@/^/@{.}[ll] &&\ar[lll]_{\alpha}
}\]
We observe that for another $\gamma\beta\in \mathbf{F}$, we necessarily have
$$[\gamma\beta].([\beta\alpha].(\alpha|q))=0.$$
It implies that the left $B$-action on $kY$ is well defined. Moreover, $kY$ is an $A$-$B$-bimodule.

\begin{lem}\label{lem:dual-A}
There is an isomorphism  $kY\simeq {\rm Hom}_A(M, A)$ of $B$-$A$-bimodules. In particular, $kY$ is a projective right $A$-module.
\end{lem}

\begin{proof}
We define a linear map
$$\phi\colon kY\longrightarrow {\rm Hom}_A(M, A)$$
such that $\phi(\alpha|q)\colon M \rightarrow A$ sends $(p|\alpha')$ to $\delta_{\alpha, \alpha'} pq$.

For each $\alpha\in Q_1$, we set $_\alpha Y$ to be the subset of $Y$ formed by elements of the form $(\alpha|q)$. We have a disjoint union $Y=\bigcup_{\alpha\in Q_1} (_\alpha Y)$ and a decomposition $kY=\bigoplus_{\alpha\in Q_1}k(_\alpha Y)$ of right $A$-modules. In view of (\ref{decom:M}), we observe that $\phi$ is a direct sum of its restriction
$$\phi_\alpha \colon k(_\alpha Y) \longrightarrow {\rm Hom}_A(kX_\alpha, A).$$

We observe a natural isomorphism of right $A$-modules
$$k(_\alpha Y)\stackrel{\sim}\longrightarrow  e_{t(\alpha)}A, \quad (\alpha|q)\mapsto q.$$
The isomorphism (\ref{iso:X}) induces an isomorphism
$${\rm Hom}_A(kX_\alpha, A)\simeq {\rm Hom}_A(Ae_{t(\alpha)}, A).$$
Using these isomorphisms, the above restriction $\phi_\alpha$ is identified with the canonical isomorphism
$$e_{t(\alpha)}A\simeq  {\rm Hom}_A(Ae_{t(\alpha)}, A).$$
This proves that $\phi$ is an isomorphism of right $A$-modules. It is routine to verify that $\phi$ respects the left $B$-actions.
\end{proof}

We consider the following subset of $Y$
\begin{equation*}
Y''=\left\{ (\alpha|\alpha q')\in Y\; \left| \; \begin{aligned} &\alpha\in Q_1 \mbox{ such that there is no arrow } \beta \mbox{ with }\alpha\beta\in \mathbf{F}, \\
& \mbox{ both } q' \mbox{ and } \alpha q' \mbox{ are nonzero paths in } A \end{aligned} \right. \right\}.
\end{equation*}
Set $Y'=Y\backslash Y''$ to be its complement. This yields a decomposition $$kY=kY'\oplus kY''$$
as a left $B$-module. Furthermore, we need the following subset
$$Y'_{\rm top}=\{(\alpha|q)\in Y\; |\; \mbox{ the nonzero path } q \mbox{ does not end with } \alpha\}\subseteq Y'.$$

The following terminologies will be useful. In a finite quiver $\Gamma$, the \emph{in-degree} of a vertex $i$ is defined to be ${\rm in}(i)=|\{\alpha\in \Gamma_1\; |\; t(\alpha)=i\}|$. A vertex $i$ is called a \emph{source}, if there is no arrow ending at $i$. A vertex $i$ is said to be \emph{left-bounded}, provided that there is a uniform bound of all the paths starting at $i$.

\begin{lem}\label{B-pd}
Keep the notation as above. Then the following statements hold.
\begin{enumerate}
\item There is an isomorphism $kY''\simeq \bigoplus S_\alpha^{\mu_\alpha}$ of $B$-modules, where $\alpha$ runs over all the sources in $\mathcal{R}$ and each multiplicity  $\mu_\alpha> 0$.
\item There is a short exact sequence of left $B$-modules
$$0\longrightarrow \Omega_B(kY') \stackrel{\rm inc}\longrightarrow \bigoplus_{(\alpha|q)\in Y'_{\rm top}} B e_{(\alpha|q)} \stackrel{\pi} \longrightarrow kY'\longrightarrow 0$$
where $e_{(\alpha|q)}=e_\alpha$ and the restriction of $\pi$ to $Be_{(\alpha|q)}$ sends $e_{\alpha|q}$ to $(\alpha|q)\in Y'$. Moreover, $\pi$ is a projective cover of $kY'$ and $\Omega_B(kY')\simeq \bigoplus S_\beta^{\nu_\beta}$, where $\beta$ runs over all the vertices in $\mathcal{R}$ satisfying ${\rm in}(\beta)\geq 2$ and each multiplicity $\nu_\beta>0$.
\end{enumerate}
\end{lem}

\begin{proof}
(1) For each $(\alpha|\alpha q')\in Y''$, the corresponding vertex $\alpha$ in $\mathcal{R}$ is a source. We observe that in the canonical decomposition $kY''=\bigoplus_{(\alpha|\alpha q')\in Y''} k(\alpha|\alpha q')$,  each direct summand  $k(\alpha|\alpha q')$ is isomorphic to $S_\alpha$. To see that each $\mu_\alpha$ is positive, we just note that the element $(\alpha|\alpha)$ does belong to  $Y''$.

(2) We observe that ${\rm rad}(B).(kY')=k(Y'\backslash Y'_{\rm top})$. In other words, the subspace $kY'_{\rm top}$ might be identified with the top of $kY'$, that is, ${\rm top}(kY')=kY'/{{\rm rad}(B).(kY')}$.  It is well known that the projective cover of $kY'$ is isomorphic to the projective cover of its top. Then we infer that $\pi$ is the desired projective cover.

Since $B$ is radical square zero, we know that  $\Omega_B(kY')$ is semisimple. We observe that the following two subsets of $\bigoplus_{(\alpha|q)\in Y'_{\rm top}} Be_{(\alpha|q)}$:
\begin{align*}
&\{ [\beta\alpha]e_{(\alpha|q)}-[\beta\alpha']e_{(\alpha'|q)} \; |\; \alpha\neq \alpha', \; (\alpha|q), (\alpha'|q)\in Y'_{\rm top}, \; \beta q\mbox{ is a nonzero path in } A  \}\\
& \mbox{ and } \; \{[\beta\alpha]e_{(\alpha|q)}\; |\; (\alpha|q)\in Y'_{\rm top}, \; \beta q=0 \mbox{ in } A\}
\end{align*}
span the kernel of $\pi$. In each subset, the corresponding vertex of $\beta$ in $\mathcal{R}$ satisfies $${\rm in}(\beta)=|\{\alpha\in Q_1\; |\; \beta\alpha\in \mathbf{F}\}|\geq 2.$$
The subspace spanned by the corresponding element $[\beta\alpha]e_{(\alpha|q)}-[\beta\alpha']e_{(\alpha'|q)}$ or $[\beta\alpha]e_{(\alpha|q)}$ is a $B$-submodule, and  is isomorphic to $S_\beta$. Then we obtain the desired decomposition of $\Omega_B(kY')$ into direct sums of $S_\beta$.

Finally, to see that each $\nu_\beta$ is positive, we assume that both $\beta\alpha$ and $\beta\alpha'$ lie in $\mathbf{F}$. Then the element $(\alpha|\alpha')$ lies in $Y'_{\rm top}$ and $[\beta\alpha]e_{(\alpha|\alpha')}$ lies in $\Omega_B(kY')$. The latter element spans a $B$-submodule of $\Omega_B(kY')$, which is isomorphic to $S_\beta$.
\end{proof}

\begin{prop}\label{prop:main}
Keep the notation as above. Then ${\rm Hom}_A(M, A)$, as a left $B$-module, has finite projective dimension if and only if any vertex in $\mathcal{R}$ is left-bounded provided that it is a source or has in-degree at least two.

In this case, there is a $B$-$A$-bimodule $N$ such that $(M, N)$ defines a singular equivalence with a certain level.
\end{prop}

\begin{proof}
By Lemma~\ref{lem:dual-A} we identify $kY$ with ${\rm Hom}_A(M, A)$. Since $B$ is radical square zero, a simple module $S_\alpha$ has finite projective dimension if and only if the vertex $\alpha$ in $\mathcal{R}$ is left-bounded. Then Lemma~\ref{B-pd} implies the first statement. The last one follows from Proposition~\ref{prop:sing}.
\end{proof}

\begin{rem}
(1) We observe that the Gorenstein cases are included in Proposition~\ref{prop:main}. This might be deduced from Lemma~\ref{lem:Gor} or the description of $\mathcal{R}$ in \cite[Proposition~5.5(1)]{CSZ}.

(2) Assume that $A$ satisfies the condition in Proposition~\ref{prop:main}. By applying  \cite[Theorems~I and II]{CLW}, the above singular equivalence with level proves that Keller's conjecture  holds for $A$; compare Remark~\ref{rem:singequi}(2).

(3) In the above singular equivalence with level defined by $(M, N)$, we do not have a concrete description of the $B$-$A$-bimodule $N$; consult the proof of Proposition~\ref{prop:sing}.
\end{rem}

\section{The $B$-dual bimodule}

Let $M$ be the $A$-$B$-bimodule defined in Section~\ref{sec:M}. In this final section, we describe the $B$-dual bimodule ${\rm Hom}_{B^{\rm op}}(M, B)$.

Recall from the proof of Lemma~\ref{lem:M} the  set
$$D=\{(p, \alpha)\in X\; |\; p\alpha \mbox{ is a nonzero path in } A\}.$$
We introduce a new set as follows
$$Z=\{(e_\alpha|p,\alpha), ([\beta\alpha]|p,\alpha)\;|\; (p, \alpha)\in D,\mbox{ and }\beta\alpha\in\mathbf{F}\}.$$
We will define a $B$-$A$-bimodule structure on $kZ$. The left $B$-action is given by the left multiplication on the leftmost entries of the triples in $Z$. For example, we have
$$[\beta\alpha] (e_\alpha|p, \alpha)=([\beta\alpha]| p, \alpha).$$
We observe that $kZ$ is a projective $B$-module.

The right $A$-action on $kZ$ is defined as follows. For each nonzero path $q$ in $A$, we set
\begin{equation}
(e_\alpha|p,\alpha).q=\begin{cases}
(e_\alpha|\gamma,\alpha)& \mbox{ if } p=q\gamma \mbox{ for a nonzero path } \gamma;\\
\sum_{\{\beta\in Q_1\; |\; \alpha\beta\in \mathbf{F}\}}\; ([\alpha\beta]|e_{t(\beta)},\beta) & \mbox{ if } q=p\alpha;\\
0 & \mbox{ otherwise}.
\end{cases}
\end{equation}
Furthermore, we set
\begin{equation}
([\beta\alpha]|p,\alpha).q=\begin{cases}
([\beta\alpha]|\gamma,\alpha)& \mbox{ if } p=q\gamma \mbox{ for a nonzero path } \gamma;\\
0 & \mbox{ otherwise}.
\end{cases}
\end{equation}

This lemma is analogous to Lemma~\ref{lem:dual-A}.

\begin{lem}\label{lem:dual-B}
There is an isomorphism  $kZ\simeq {\rm Hom}_{B^{\rm op}}(M, B)$ of $B$-$A$-bimodules.
\end{lem}

\begin{proof}
We define a $k$-linear map
$$\psi\colon kZ \longrightarrow {\rm Hom}_{B^{\rm op}}(M, B)$$
such that $\psi(x|p, \alpha)\colon M\rightarrow B$ is the unique right $B$-module morphism sending $(p', \alpha')\in D$ to $\delta_{p, p'}\delta_{\alpha, \alpha'} x $ for $x\in\{e_\alpha, [\beta\alpha]\; |\;\beta\in Q_1,  \beta\alpha\in \mathbf{F}\}$. In view of the decomposition (\ref{decom:M2}), it is not hard to prove that $\psi$ is an isomorphism of left $B$-modules. We omit the verification that it respects the right $A$-actions.
\end{proof}

For each $\alpha\in Q_1$, we consider the following subset of $Z$
$${_\alpha}Z=\{(e_\alpha|p,\alpha) \mid  (p, \alpha)\in D\} \cup \{([\alpha\beta]|q,\beta)\mid \alpha\beta\in \mathbf{F}, \; (q, \beta)\in D\}.$$
We have a disjoint union $Z=\bigcup_{\alpha\in Q_1}{_\alpha Z}$, which yields a decomposition of right $A$-modules
\begin{equation*}
kZ= \bigoplus_{\alpha\in Q_1} k({_\alpha Z}).
\end{equation*}
Lemma~\ref{lem:dual-B} and this decomposition might be useful to study the problem when ${\rm Hom}_{B^{\rm op}}(M, B)$ has finite projective dimension as a right $A$-module.

The following example shows that this case is quite different from the $A$-dual bimodule.

\begin{exm}
	Let $A$ be given by the following quiver $Q$:
		\[\xymatrix{
		1\ar@/^0.6pc/[r]^{\alpha}&2\ar@/^0.6pc/[l]^{\beta}\ar@(ur,dr)^{\gamma}& }\]
	with relations  given by $\mathbf{F}=\{\alpha\beta,\beta\gamma,\gamma\gamma\}$. The relation quiver  $\mathcal{R}$ is
	\[\xymatrix{ \gamma\ar@(ul,dl)_{[\gamma\gamma]}\ar[r]^{[\beta\gamma]}&\beta\ar[r]^{[\alpha\beta]}& \alpha}\]
	We observe that $A$ is non-Gorenstein by \cite[Proposition~5.5(1)]{CSZ}.

In $\mathcal{R}$, there is neither sources nor vertices with in-degree at least two. It follows that the $A$-dual bimodule ${\rm Hom}_A(M,A)$ is projective as a left $B$-module by Lemma~\ref{B-pd}.

	We observe that
$${_\gamma Z}=\{(e_{\gamma}|e_2,\gamma), ([\gamma\gamma]|e_2,\gamma)\}.$$
  The projective cover of $k({_\gamma Z})$ is $e_2A$, and its syzygy is isomorphic to a direct sum of two copies of $S_1$. It follows that as a right $A$-module, $k({_\gamma Z})$ has infinite projective dimension. We conclude that  as a right $A$-module, ${\rm Hom}_{B^{\rm op}}(M, B)$ has infinite projective dimension.
\end{exm}

In view of the following known cases, the non-Gorensteinness in the previous example is essential.

\begin{lem}
Assume that $A$ is Gorenstein. Then ${\rm Hom}_{B^{\rm op}}(M, B)$ has finite projective dimension as a right $A$-module.
\end{lem}

\begin{proof}
As we mention in the proof of Theorem~\ref{thm:main}, the algebra $B$ is also Gorenstein. Then the result is a special case of Lemma~\ref{lem:Gor}.
\end{proof}

However, we do not have nice conditions on when precisely ${\rm Hom}_{B^{\rm op}}(M, B)$ has finite projective dimension as a right $A$-module.

\vskip 10pt

\noindent {\bf Acknowledgements.}\quad The first author thanks Bernhard Keller and Zhengfang Wang for helpful discussion. This work is supported by  National Natural Science Foundation of China (No.s 11671245,11971449, and 11901551).

\bibliography{}

\vskip 10pt

 {\footnotesize \noindent Xiao-Wu Chen, Jian Liu, Ren Wang\\
 Key Laboratory of Wu Wen-Tsun Mathematics, Chinese Academy of Sciences,\\
 School of Mathematical Sciences, University of Science and Technology of China, Hefei 230026, Anhui, PR China\\
E-mail: xwchen@mail.ustc.edu.cn; liuj231@mail.ustc.edu.cn; renw@mail.ustc.edu.cn}

\end{document}